\documentclass[12pt,leqno]{article}
\usepackage{amsmath, amssymb, amsfonts, amsthm,amscd}
\usepackage{eucal}

\swapnumbers
\theoremstyle{plain}
\newtheorem{theorem}{Theorem}[section]
\newtheorem{lemma}[theorem]{Lemma}
\newtheorem{Titletheo}[theorem]{}

\newtheorem{corollary}[theorem]{Corollary}

\theoremstyle{definition}
\newtheorem{remark}[theorem]{Remark}

\numberwithin{equation}{theorem}
\newcommand{\newno}{{\stepcounter{theorem}}{\thetheorem}}
\newcommand{\subhead}[1]{\medskip
 \ni {\bf \newno\; {#1}}}

\renewcommand{\thetheorem}{(\arabic{section}.\arabic{theorem})}

  \newcommand{\om}{\omega}    \newcommand{\Om}{\Omega}
  \newcommand{\sig}{\sigma}   
  \newcommand{\al}{\alpha}
  \newcommand{\del}{\delta}  \newcommand{\Del}{\Delta}
  \newcommand{\gam}{\gamma}   
  \newcommand{\lam}{\lambda}   
  \newcommand{\eps}{\varepsilon}
   \def\part{\partial}
  \newcommand{\wed}{\wedge}
  \newcommand{\bsym}{\boldsymbol}
  
   \newcommand{\rtsquig}{\rightsquigarrow}

  \def\alv{\alpha^{\vee}}
  \def\avi{\alpha^{\vee}_i}

  \def\b1{\text{\bf\large 1}}  

  \def\simto{\overset{\sim}{\longrightarrow}}
  
  \def\ip<#1>{\langle#1\rangle}   

  \newcommand{\Imo}{\operatorname{Im}}
  
  \newcommand{\Ker}{\operatorname{Ker}}
  
  \newcommand{\minu}{\operatorname{minu}}

\newcommand{\beqn}{\begin{equation}}
\newcommand{\eeqn}{\end{equation}}

\newcommand{\x}{\times}

\renewcommand{\ni}{\noindent}

 \newcommand{\fb}{\mathfrak{b}}
 \newcommand{\fg}{\mathfrak{g}}
 \newcommand{\fh}{\mathfrak{h}}
 \newcommand{\fk}{\mathfrak{k}}
 \newcommand{\fn}{\mathfrak{n}}
 \newcommand{\fp}{\mathfrak{p}}
 \newcommand{\fr}{\mathfrak{r}}
 \newcommand{\fs}{\mathfrak{s}}
 \newcommand{\fu}{\mathfrak{u}}

 \newcommand{\fy}{\mathfrak{y}}
\newcommand{\fri}{\mathfrak{i}}

\newcommand{\bc}{\mathbb{C}}

\newcommand{\bz}{\mathbb{Z}}

 \newcommand{\cb}{\mathcal{B}}
 \newcommand{\cd}{\mathcal{D}}
 
 \newcommand{\cg}{\mathcal{G}}

 \def\cp{\mathcal{P}}
 \def\cw{\mathcal{W}}
 \def\cx{\mathcal{X}}
  \newcommand{\cy}{\mathcal{Y}}

\begin{document}

  \title{A Generalization of Cachazo-Douglas-Seiberg-Witten Conjecture
for Symmetric Spaces}
  \author{Shrawan Kumar\\Department of Mathematics\\
University of North Carolina\\
Chapel Hill, NC  27599--3250}
  \maketitle

  \section{Introduction}
Let $\fg$ be a (finite-dimensional) semisimple Lie algebra over the
complex numbers $\bc$ and let $\sig$ be an involution (i.e., an
automorphism of order 2) of $\fg$.  Let $\fk$ (resp. $\fp$) be the +1
(resp. $-1$) eigenspace of $\sig$.  Then, $\fk$ is a Lie subalgebra
of $\fg$ and $\fp$ is a $\fk$-module under the adjoint action.  In
this paper we only consider those involutions $\sig$ such that $\fp$
is an irreducible $\fk$-module.

We fix a $\fg$-invariant nondegenerate symmetric bilinear form
$\ip<\,\, , \,\,>$ on $\fg$.  Then, the decomposition
  \[
\fg = \fk\oplus\fp
  \]
is an orthogonal decomposition.

Let $R := \wed (\fp\oplus\fp )$ be the exterior algebra on two
copies of $\fp$.  To distinguish, we denote the first copy of $\fp$
by $\fp_1$ and the second copy by $\fp_2$.  It is bigraded by declaring
$\fp_1$ (resp. $\fp_2$) to have bidegree (1,0) (resp. (0,1)).  Choose any basis $\{
e_i\}$ of $\fp$ and let $\{ f_i\}$ be the dual basis of $\fp$,
i.e.,
  \[
\ip< e_i,f_j> = \del_{i,j} .
  \]

Define a $\fk$-module map (under the adjoint action)
  \[
c_3: \fk \to \fp\otimes \fp , \qquad c_3(x) = \sum_i [x, e_i]\otimes
f_i .
  \]

It is easy to see that $c_3$ does not depend upon the choice of the
basis $\{ e_i\}$.  Projected onto $\wed^2(\fp )$, we get a
$\fk$-module map $\fk\to\wed^2(\fp )$.  This map is denoted by
$c_1$ considered as a map $\fk\to\wed^2(\fp_1)$, and similarly for
$c_2: \fk\to\wed^2(\fp_2)$.  We denote the image of $c_i$ by $C_i$.
Let $J$ be the (bigraded) ideal of $R$ generated by $C_1\oplus
C_2\oplus C_3$ and let us consider the quotient algebra
  \[
A := R/J .
  \]

  The algebra $A$ is a $\fk$-algebra (induced from the adjoint
action of $\fk$) and let $A^{\fk}$ be the subalgebra of
$\fk$-invariants.
The algebra $A^{\fk}$ contains the element $S := \sum e_i\otimes
f_i$ in bidegree (1,1).

{\it The aim of this paper is to understand the structure of the algebra  $A^{\fk}$.}

In the case when $\fg=\mathfrak s\oplus \mathfrak s$ for a simple Lie algebra $\mathfrak s$
and $\sigma$ is the involution which switches the two factors, the study of the structure of
$A^{\fk}$ was initiated  by Cachazo-Douglas-Seiberg-Witten who made the following conjecture.
(Observe that in this case $\fk$ and $\fp$ both can be identified with $\mathfrak s$ and the
adjoint action of  $\fk$ on $\fp$ under this identification is nothing but the adjoint action of
$\mathfrak s$ on itself.) We will refer to this as the {\it diagonal case}.

\begin{Titletheo}{\rm\bf Conjecture [CDSW]} (i) The subalgebra $A^{\fk}$ of
$\fk$-invariants in $A$ is generated, as an algebra, by the element $S$.

(ii) $S^h =0$.

(iii) $S^{h-1}\neq 0$,

\noindent
where $h$ is the dual Coxeter number of $\fk=\fs$.
\end{Titletheo}

They
proved the conjecture for  $\fs=sl_N$ in  [CDSW], and Witten proved it for
$\fs=sp_N$ in [W].  He also proved parts (i) and (ii) of the conjecture for
$\fs=so_N$ in [W]. Subsequently,
Etingof-Kac proved the conjecture for $\fs$ of type $G_2$ by using the
theory of abelian ideals. Kumar proved part (i) of the conjecture
uniformly in [K3] using geometric and topological methods.

Returning to the general case of any involution $\sigma$, we prove the following analogous result
(cf. Theorem 4.8) which is the main result of this paper.

\begin{Titletheo} {\rm\bf Theorem} Let $\sigma$ be any involution of a simple Lie algebra $\fg$ such that
$\fp$ is an irreducible module under the adjoint action of $\fk$. Then, the subalgebra $A^{\fk}$ of
$\fk$-invariants in $A$ is generated, as an algebra, by the element $S$.
\end{Titletheo}

Analogous to our proof in the diagonal case, we need to consider the algebra
$B:=R/\ip<C_1\oplus C_2>$.
We show (cf. Theorem 3.1) that the subalgebra $B^\fk$ of  $\fk$-invariants of $B$
is graded isomorphic with the singular cohomology with complex coefficients $H^*(\cy)$ of a
certain finite-dimensional projective subvariety $\cy$ of the twisted affine Grassmannian
$\cx_{\sig} $ (cf. Section 2 for the definitions of $\cx_{\sig} $ and $\cy$). The definition
of the subvariety $\cy$ is motivated from the theory of abelian subspaces
of $\fp$. The main ingredients in our proof of Theorem 3.1 are: result of Garland-Lepowsky on
the Lie algebra cohomology of the nil-radical $\hat{\fu}_{\sig}$ of a maximal parabolic subalgebra of twisted
affine Kac-Moody Lie algebras;
the `diagonal'
cohomology of $\hat{\fu}_{\sig}$ introduced by Kostant; certain results of Han and
Cellini-Frajria-Papi on abelian subspaces of $\fp$
and a certain deformation of the singular cohomology of  $\cx_{\sig} $ introduced by
Belkale-Kumar.

Having identified the algebra $B^\fk$ with $H^*(\cy)$, we next use the fact that
$H^*(\cx_{\sig}) $ surjects onto $H^*(\cy)$ under the restriction map. Section 4 is devoted to
 study the cohomology algebra $H^*(\cx_{\sig}) $. The results here are more involved than
 in the diagonal case. One mazor difficulty arises from the fact that the fibration
 \[
\Om^{\sig}_1(G_o) \to \Om^{\sig}(G_o)/K_o \overset{\gamma}\longrightarrow G_o/K_o ,
  \]
  is nontrivial (cf. Section 4 for various notation). To complete the proof of our
  Theorem 4.8,
  we show that all but one of the generators of $H^*(\cx_{\sig}) $ go to zero
  under the canonical
  projection map $B^\fk \to A^\fk$ and the remaining one generator goes to $S$.

 Finally, analogous to the Cachazo-Douglas-Seiberg-Witten Conjecture, we make
the following conjecture.

\begin{Titletheo}{\rm\bf Conjecture}
$S^h=0$ and $S^{h-1}\neq 0$ in $A^{\fk}$, where $h=h_{\fg}
-h_{\fk}$ ($h_{\fg}$ being the dual Coxeter number of $\fg$).
\end{Titletheo}

Unless otherwise stated, by the cohomology $H^*(X)$ of a topological
space $X$ we mean the singular cohomology $H^*(X,\bc )$ with complex
coefficients.

\vskip2ex

\noindent
{\it Acknowledgements.} It is my pleasure to thank Weiqiang Wang for asking the
question answered in this paper. I also thank Pierluigi Frajria and Paolo Papi for some helpful
correspondences.
In particular, the proof of Lemma 2.4 is due to them.
This work was partially supported by the FRG grant no DMS-0554247 from NSF.

\section{Preliminaries and Notation}

  \subhead{Twisted affine Lie algebras.}
Let $\fg$ be a (finite-dimensional) simple Lie algebra over $\bc$
and let $\sig$ be an involution of $\fg$.  Let $\fk\subset\fg$ be
the $+1$ eigenspace of $\sig$ (which is a reductive subalgebra of
$\fg$) and let $\fp$ be the $-1$ eigenspace of $\sig$, which is a
$\fk$-module under the adjoint action.  As in the introduction, we
only consider those involutions $\sig$ such that $\fp$ is an
irreducible $\fk$-module.  {\em This will be our tacit assumption on
$\sig$ throughout the paper.}

Fix a Cartan subalgebra $\fh_{\sig}$ and a Borel subalgebra
$\fb_{\sig} \supset \fh_{\sig}$ of $\fk$.  Let $\fn_{\sig}$ be the
nil-radical of $\fb_{\sig}$.  Associated to the pair $(\fg ,\sig )$
we have the {\em twisted affine Kac-Moody Lie algebra}
  \[
\hat{\fg}_{\sig} := \sum_{i\in\bz} \fg_i\otimes t^i\oplus \bc
c\oplus \bc d,
  \]
where $\fg_{2i} :=\fk$ and $\fg_{2i+1} :=\fp$ for any $i\in\bz$.
The bracket in $\hat{\fg}_{\sig}$ is defined as follows:
  \begin{gather*}
\bigl[ x\otimes t^m +\lam c+\mu d, x'\otimes t^{m'} +\lam' c+\mu'
d\bigr] =\\
\bigl( [x,x']\otimes t^{m+m'} + \mu m' x'\otimes t^{m'} -\mu'
mx\otimes t^m\bigr) +m\, \del_{m,-m'} \ip<x,x'>c,
  \end{gather*}
where $\ip<\,\, ,\,\,>$ is the normalized $\fg$-invariant bilinear
form on $\fg$ as in the introduction.

The Lie algebra $\hat{\fg}_{\sig} $ is a subalgebra of the affine Kac-Moody algebra
\[\hat{\fg}:=\sum_{i\in\bz} \fg\otimes t^i\oplus \bc
c\oplus \bc d\]
 with the bracket defined by the same formula as above.

We define the following subalgebras of $\hat{\fg}_{\sig}$ called the
{\em standard Cartan}, {\em standard Borel} and the {\em standard
maximal parabolic subalgebra} respectively:
  \begin{align*}
\hat{\fh}_{\sig} &:= \fh_{\sig}\otimes t^0\oplus\bc c\oplus\bc d,\\
\hat{\fb}_{\sig} &:= \fb_{\sig}\otimes t^0\oplus \sum_{i>0}
\fg_i\otimes t^i  \oplus\bc c\oplus\bc d, \text{ and}\\
\hat{\fp}_{\sig} &:= \sum_{i\geq 0} \fg_i\otimes t^i \oplus \bc
c\oplus\bc d.
  \end{align*}

We also have the {\em nil-radicals} $\hat{\fn}_{\sig}$ of
$\hat{\fb}_{\sig}$ and $\hat{\fu}_{\sig}$ of $\hat{\fp}_{\sig}$ and
the Levi subalgebra $\hat{\fr}_{\sig}$ of $\hat{\fp}_{\sig}$
defined as follows:
  \begin{align*}
\hat{\fn}_{\sig} &:= \fn_{\sig}\otimes t^0\oplus \sum_{i>0}
\fg_i\otimes t^i, \\
\hat{\fu}_{\sig} &:= \sum_{i>0} \fg_i\otimes t^i, \text{ and}\\
\hat{\fr}_{\sig} &:= \fk\otimes t^0\oplus\bc c\oplus\bc d.
  \end{align*}

The evaluation at 1 gives rise to a Lie algebra homomorphism
  \[
ev_1: \hat{\fg}_{\sig} \to \fg\oplus\bc c\oplus\bc d,
  \]
where $c$ and $d$ are central in the right side.

Associated to the twisted affine Kac-Moody Lie algebra
$\hat{\fg}_{\sig}$ and its subalgebras $\hat{\fp}_{\sig}$ and
$\hat{\fb}_{\sig}$, we have the twisted affine Kac-Moody group
$\cg_{\sig}$, the standard maximal parabolic subgroup $\cp_{\sig}$
and the standard Borel subgroup $\cb_{\sig}$ respectively (cf. [K2,
Chapter 6]).

Let $W_{\sig}$ be the (finite) Weyl group of $(\fk ,\fh_{\sig})$ and
let $\cw_{\sig}$ be the (affine) Weyl group of $\bigl(
\hat{\fg}_{\sig}, \hat{\fh}_{\sig}\bigr)$.  Let $\hat{\Del}^+_{\sig}
\subset (\hat{\fh}_{\sig})^*$ be the set of positive roots of
$\hat{\fg}_{\sig}$, i.e., the set of roots for the subalgebra
$\hat{\fn}_{\sig}$ with respect to the adjoint action of
$\hat{\fh}_{\sig}$.  We set $\hat{\Del}^-_{\sig} =
-\hat{\Del}^+_{\sig}$.  For any $w\in\cw_{\sig}$, define
  \begin{align*}
\Phi (w) &:= \hat{\Del}^+_{\sig} \cap w\ \hat{\Del}^-_{\sig} , \text{
and}\\
\hat{\fn}_{\sig}(w) &:= \bigoplus_{\al\in\Phi (w)}
\bigl(\hat{\fg}_{\sig}\bigr)_{\al},
  \end{align*}
where $\bigl(\hat{\fg}_{\sig}\bigr)_{\al}$ denotes the root space of
$\hat{\fg}_{\sig}$ corresponding to the root $\al$.  Since each root
in $\Phi (w)$ is real, $\bigl(\hat{\fg}_{\sig}\bigr)_{\al}$ is
one-dimensional for each $\al\in\Phi (w)$.

  \subhead{Abelian subspaces of $\fp$.}
Let $\cw'_{\sig}\subset\cw_{\sig}$ be the set of minimal coset
representatives in the cosets $\cw_{\sig}/W_{\sig}$.

Following [CFP], we call an element $w\in\cw_{\sig}$ {\em minuscule}
if
  \[
\hat{\fn}_{\sig}(w^{-1}) \subset \fp\otimes t.
  \]

Let us denote the set of minuscule elements in $\cw_{\sig}$ by
$\cw^{\minu}_{\sig}$.  Then, it is easy to see that
$\cw^{\minu}_{\sig} \subset \cw'_{\sig}$ and, clearly, it is a
finite set.

We recall the following result from [CFP, Theorem 3.1].

  \begin{theorem}  There is a bijection between $\cw^{\minu}_{\sig}$
and the set $\Xi$ of $\fb_{\sig}$-stable abelian subspaces of $\fp$
given by $w\mapsto ev_1(\hat{\fn}_{\sig}(w^{-1}))$.  In particular,
the cardinality $|\cw^{\minu}_{\sig}| = |\Xi |$.
  \end{theorem}

We recall the Bruhat decomposition (cf. [K2, Corollary 6.1.20]) of
the projective ind-variety
  \[
\cx_{\sig} := \cg_{\sig}/\cp_{\sig} = \bigsqcup_{w\in\cw'_{\sig}}
\cb_{\sig}w \cp_{\sig}/\cp_{\sig} ,
  \]
where the {\em Bruhat cell} $C(w) :=
\cb_{\sig}w\cp_{\sig}/\cp_{\sig}$ is isomorphic to the affine space
$\bc^{\ell (w)}$ ($\ell (w)$ being the length of $w$ in the Coxeter
group $\cw_{\sig}$).  Moreover, for any $w\in\cw'_{\sig}$, the
Zariski closure
  \[
\overline{C(w)} = \bigsqcup_{\substack{ v\in\cw'_{\sig}\text{ and}\\
v\leq w\qquad}} C(v).
  \]

Define a subset $\cy$ of $\cg_{\sig}/\cp_{\sig}$ by
  \[
\cy = \bigsqcup_{w\in\cw^{\minu}_{\sig}} C(w).
  \]

Then, $\cy$ is a (finite-dimensional) projective subvariety of
$\cg_{\sig}/\cp_{\sig}$.  This follows from the following.

  \begin{lemma}  For $w\in\cw^{\minu}_{\sig}$ and any
$u\in\cw'_{\sig}$ such that $u\leq w$, we have
$u\in\cw^{\minu}_{\sig}$.
  \end{lemma}

  \begin{proof}[Proof {\em (due to P. Frajria and P. Papi)}]
  By the definition,
an element $u\in\cw_{\sig}$ is minuscule iff $\beta (d)=1$ for all
$\beta\in\Phi (u^{-1})$.  By the $L$-shellability of the Bruhat
order in $\cw'_{\sig}$, we can assume that $w=us_{\al}$, where
$\al\in\hat{\Del}^+_{\sig}$ is a real root and $s_{\al}$ is the
reflection through $\al$: $s_{\al}\lam = \lam -\ip<\lam ,\alv> \al$
for $\lam\in (\hat{\fh}_{\sig})^*$.  Since $u<w$, we have
$w\al\in\hat{\Del}^-_{\sig}$, and hence $\al\in\Phi (w^{-1})$.  In
particular, $\al (d)=1$.  Since $u\in\cw'_{\sig}$, we have $\beta
(d)\neq 0$ for any $\beta\in\Phi (u^{-1})$.  Thus, it suffices to prove that
 for any $\beta\in
\hat{\Del}^+_{\sig}$ such that $\beta (d)>1$, we have
$u\beta\in\hat{\Del}^+_{\sig}$.  Observe that since $\beta (d)>1$,
$w\beta\in\hat{\Del}^+_{\sig}$.

There are three cases to consider:
  \begin{description}
\item[{\em Case I}:]  $s_{\al}\beta\in\hat{\Del}^-_{\sig}$.\\
In this case, $\ip<\beta ,\alv> >0$.  Thus,
  \[
u\beta = w(s_{\al}\beta ) = w(\beta -\ip<\beta ,\alv> \al) = w\beta
-\ip<\beta ,\alv> w\al\in\hat{\Del}^+_{\sig} ,
  \]
  since $w\al\in \hat{\Del}^-_{\sig}$.
\item[{\em Case II}:] $s_{\al}\beta\in\hat{\Del}^+_{\sig}$ and
$s_{\al}\beta (d)\neq 1$.\\
In this case, $s_{\al}\beta\notin\Phi (w^{-1})$, i.e., $u\beta =
ws_{\al}\beta \in\hat{\Del}^+_{\sig}$.
\item[{\em Case III}:] $s_{\al}\beta\in\hat{\Del}^+_{\sig}$ and
$s_{\al}\beta (d)=1$.\\
In this case,
  \begin{align*}
  s_{\al}\beta (d) &= \beta (d) -\ip<\beta ,\alv>\, \al (d)\\
  &= \beta (d) -\ip<\beta ,\alv> =1,\,\,\text{since}\, \al (d)=1.
  \end{align*}

  \end{description}
Thus, $\ip<\beta ,\alv> = \beta (d)-1 > 0$ (since $\beta (d)>1$) and
hence $u\beta =ws_{\al}\beta = w\beta -\ip<\beta ,\alv >
(w\al )\in\hat{\Del}^+_{\sig}$, since $w\al\in\hat{\Del}^-_{\sig}$.
This proves the lemma.
  \end{proof}

 \section{Topological identification of the algebra $B^{\bsym{\fk}}$}

Consider the $\bz_+$-graded $\fk$-algebra
  \[
B := \frac{\wed (\fp )\otimes \wed (\fp )}{\ip< C_1\oplus C_2>},
  \]
where $C_1$ and $C_2$ are defined in the Introduction.

Following is the first main result of this paper.

  \begin{theorem}  The singular cohomology $H^*(\cy ,\bc )$ of $\cy$ with complex coefficients is isomorphic as a $\bz_+$-graded algebra with the graded
algebra of $\fk$-invariants $B^{\fk}$.
  \end{theorem}

Before we come to the proof of the theorem, we need to recall the following
results.  The first theorem is a special case of a result due to Garland-
Lepowsky and the second theorem is due to Han.

  \begin{theorem} {\em [K2, Theorem 3.2.7 and Identity (3.2.11.3)]}
As a module for $\hat{\fr}_{\sig}$,
  \[
H^p(\hat{\fu}_{\sig},\bc ) \simeq \bigoplus_{ \substack{w\in\cw'_{\sig}\\
\ell (w)=p}} L(w^{-1}\hat{\rho}-\hat{\rho}) ,
  \]
where $\hat{\rho}$ is any element of $\bigl( \hat{\fh}_{\sig}\bigr)^*$
satisfying $\hat{\rho}(\avi )=1$ for all the simple coroots $\{\alv_0 ,
\dots ,\alv_{\ell}\} \subset \hat{\fh}_{\sig}$ of $\hat{\fg}_{\sig}$
and $L(w^{-1}\hat{\rho}-\hat{\rho})$ denotes the irreducible
$\hat{\fr}_{\sig}$-module with highest weight $w^{-1}\hat{\rho}-
\hat{\rho}$.  Similarly, by [K2, Theorem 3.2.7],
  \[
H^p(\hat{\fu}_{\sig}^-,\bc ) \simeq \bigoplus_{\substack{ w\in\cw'_{\sig},\\
\ell (w)=p }} L(w^{-1}\hat{\rho}-\hat{\rho})^* ,
  \]
where $\hat{\fu}^-_{\sig} := \sum_{i<0} \fg_i\otimes t^i$.
   \end{theorem}

For any $\fb_{\sig}$-stable abelian subspace $I\subset\fp$ of dimension $n$,
$\wed^n(I)$ is a $\fb_{\sig}$-stable line in $\wed^n(\fp )$ and hence generates an irreducible $\fk$-submodule $V_I$ of $\wed^n(\fp )$ with highest weight
space $\wed^n(I)$.  Thus, we get a $\fk$-module map
  \[
\bigoplus_{I\in\Xi} V_I \to \wed (\fp ) \to \wed (\fp )/\ip<C_1> .
  \]
 If $I$ corresponds via Theorem 2.3 to the element $w\in\cw^{\minu}_{\sig}$,
 then $V_I$ has highest weight $(w^{-1}\hat{\rho} -\hat{\rho})_{|_{\fh_{\sig}}}$.

  \begin{theorem}  {\em [H, Theorem 4.7]}  The above $\fk$-module map
  \[
\bigoplus_{I\in\Xi} V_I  \to \wed (\fp )/\ip<C_1>
  \]
is an isomorphism.  Moreover, by [P, Theorem 4.13(2)], the $\fk$-module
$\bigoplus_{I\in\Xi} V_I$ is multiplicity free.
  \end{theorem}

For any $w\in\cw'_{\sig}$, define the Schubert cohomology class
 $\eps^w\in H^{2\ell (w)}(\cx_{\sig},\bz )$ by
  \[
\eps^w\bigl( [\overline{C(u)}]\bigr) = \del_{w,u} \text{  for }u\in\cw'_{\sig} ,
  \]
where $[\overline{C(u)}]\in H_{2\ell (u)}(\cx_{\sig},\bz )$ denotes the fundamental homology class of $\overline{C(u)}$.

Following Belkale-Kumar [BK, \S6], we define a new product $\odot_0$ in
$H^*(\cx_{\sig},\bz )$ as follows.  Express the standard cup product
  \[
\eps^u\cdot\eps^v = \sum_{w\in\cw'_{\sig}} c^w_{u,v}\eps^w .
  \]
Now, define
  \[
\eps^u\odot_0\eps^v = \sum\, c^w_{u,v}\, \del_{d^w_{u,v},0} \,\eps^w,
  \]
where
  \[
d^w_{u,v} := \bigl( u^{-1}\hat{\rho}+v^{-1}\hat{\rho}  -w^{-1}\hat{\rho} -\hat{\rho} \bigr) (d).
  \]

The product $\odot_0$ descends to a product in $H^*(\cy ,\bz )$ under the restriction map $H^*(\cx_{\sig},\bz )\to H^*(\cy ,\bz )$.

  \begin{lemma}  The product $\odot_0$ coincides with the standard cup
product in $H^*(\cy ,\bz )$.
  \end{lemma}

  \begin{proof}  For any $w\in\cw_{\sig}$, by [K2, Corollary 1.3.22],
  \[
|\Phi (w)| = \hat{\rho} -w\hat{\rho} ,
  \]
where
  \[
|\Phi (w)| := \sum_{\beta\in\Phi (w)} \beta.
  \]
Thus, for any $w\in\cw^{\minu}_{\sig}$, by its definition
  \beqn
(\hat{\rho} -w^{-1}\hat{\rho} )(d) =\ell (w).
  \eeqn
To prove the lemma, it suffices to show that whenever $c^w_{u,v}\neq 0$ for $u,v,w\in \cw_{\sig}^{\minu}$, $d^w_{u,v}=0$.  But, $c^w_{u,v}\neq 0$ gives
  \beqn
\ell (w) = \ell (u) +\ell (v).
  \eeqn
Thus,
  \begin{align*}
d^w_{u,v} &= \Bigl( u^{-1}\hat{\rho} -\hat{\rho} +v^{-1}\hat{\rho} -\hat{\rho} -\bigl( w^{-1}\hat{\rho}
-\hat{\rho}\bigr) \Bigr)\, (d)\\
&= -\ell (u) -\ell (v) +\ell (w)\\
&= 0 \qquad\text{  by (2)}.
  \end{align*}
  \end{proof}

 \setcounter{equation}{0}
\begin{proof}[Proof of Theorem 3.1]  The cohomology modules
$H^p(\hat{\fu}_{\sig})$ and $H^p(\hat{\fu}^-_{\sig})$ acquire a grading coming from the
 total degree of $t$ in $\wed^p(\hat{\fu}_{\sig})$ and $\wed^p(\hat{\fu}^-_{\sig})$ respectively.  This decomposes
  \[
H^p(\hat{\fu}_{\sig}) = \bigoplus_{m\in\bz_+} H^p_{(-m)} (\hat{\fu}_{\sig}),
  \]
where $H^p_{(-m)}(\hat{\fu}_{\sig})$ denotes the space of elements of
$H^p(\hat{\fu}_{\sig})$ of total $t$-degree $-m$.  Define the diagonal
cohomology
  \[
H^*_D(\hat{\fu}_{\sig}) :=  \bigoplus_{p\in\bz_+} H^p_{(-p)}
(\hat{\fu}_{\sig}),
  \]
which is a subalgebra of $H^*(\hat{\fu}_{\sig})$, and similarly define $H^*_D(\hat{\fu}^-_{\sig})$.

Let $\phi : \wed^p(\fp ) \to H^p_{(-p)}(\hat{\fu}_{\sig})$ be the map induced
from the map $\bar{\phi}: \wed^p(\fp )\to C^p_{(-p)}(\hat{\fu}_{\sig})$,
  \[
\bar{\phi}\bigl( x_1\wed\cdots\wed x_p\bigr) \bigl( y_1\otimes t\wed\cdots\wed y_p\otimes t\bigr) = \det \bigl(\ip< x_i,y_j>\bigr)_{i,j} ,
  \]
(for $x_i,y_j\in\fp$) by taking the cohomology class of the image.  Clearly,
 $\bar{\phi}(x_1\wed\cdots\wed x_p)$ is a cocycle and, moreover,
 $\bar{\phi}$ (and hence $\phi$) is surjective.  It is easy to see that
 $\Ker \bigl({\phi}_{|_{\wed^2(\fp )}}\bigr) = C_1$. Now, take any $\om\in C^{p-1}_{(-p)}
(\hat{\fu}_{\sig})$, where $C^{p-1}_{(-p)}(\hat{\fu}_{\sig})$ denotes the
space of $(p-1)$-cochains on $\hat{\fu}_{\sig}$ with total $t$-degree $-p$.  We can
write
  \[
\om =\sum^N_{i=1} \om^i_1\wed \om^i_2,
  \]
for some $\om^i_1\in C^1_{(-2)}(\hat{\fu}_{\sig})$
and $\om_2^i \in C^{p-2}_{(-p+2)}(\hat{\fu}_{\sig})$.  Then,
  \[
\del\om = \sum^N_{i=1} (\del \om^i_1) \wed \om^i_2,
  \]
since $\om^i_2$ are $\del$-closed, where $\del$ is the standard differential of the cochain complex $C^*(\hat{\fu}_{\sig})$.

From this it is easy to see that $\Ker \phi = \ip< C_1>$.  Thus, we get a graded algebra isomorphism commuting with the $\fk$-module structures:
  \beqn
\frac{\wed^*(\fp )}{\ip< C_1>} \simeq H^*_D(\hat{\fu}_{\sig}) .
  \eeqn
In exactly the same way, we get an isomorphism of graded algebras
commuting with the $\fk$-module structures:
  \beqn
  \frac{\wed^*(\fp )}{\ip< C_2>} \simeq H^*_D(\hat{\fu}^-_{\sig}).
  \eeqn

In particular, $\frac{\wed^p(\fp )}{\ip< C_1>\cap\wed^p(\fp )}$ is a self-dual $\fk$-module
 for any $p\geq 0$.

Combining (1)--(2), we get an isomorphism (for any $p,q \geq 0$)
  \beqn
\biggl[ \frac{\wed^p(\fp )}{\ip< C_1>\cap\wed^p(\fp )} \otimes
\frac{\wed^q(\fp )}{\ip< C_2>\cap\wed^q(\fp )}\biggr]^\fk
\simeq \Bigl[ H^p_D(\hat{\fu}_{\sig}) \otimes
H^q_D(\hat{\fu}^-_{\sig})\Bigr]^{\fk} .
  \eeqn
Since $\frac{\wed^*(\fp )}{\ip< C_1>}$ is multiplicity free (by Theorem 3.3)
 and $\frac{\wed^p(\fp )}{\ip< C_1>\cap\wed^p(\fp )}$ is self-dual
 for any $p\geq 0$, the left side of (3) is nonzero only if $p=q$.
 Moreover, $c$ acts trivially on $H^p_D(\hat{\fu}_{\sig})\otimes H^q_D(\hat{\fu}^-_{\sig})$
 and $d$ acts via the multiplication by $q-p$.  Thus, we have a graded algebra isomorphism:
  \beqn
\biggl[\frac{\wed^*(\fp )}{\ip< C_1>} \otimes
\frac{\wed^*(\fp )}{\ip< C_2>}\biggr]^\fk
\simto \Bigl[ H^*_D(\hat{\fu}_{\sig}) \otimes
H^*_D(\hat{\fu}^-_{\sig})\Bigr]^{\hat{\fr}_{\sig}} .
  \eeqn

By Theorem 3.2, we get
  \beqn
H^p_D(\hat{\fu}_{\sig}) \simeq H^p_D(\hat{\fu}^-_{\sig})^* \simeq
\bigoplus_{\substack{w\in\cw^{\minu}_{\sig}\\ \ell (w)=p}}
L\bigl( w^{-1}\hat{\rho}-\hat{\rho}\bigr),
  \eeqn
as $\hat{\fr}_{\sig}$-modules.  Combining (4)--(5), we get the isomorphism
  \beqn
\biggl[ \frac{\wed^*(\fp )}{\ip< C_1>} \otimes
\frac{\wed^*(\fp )}{\ip< C_2>}\biggr]^{\fk} \simeq
\bigoplus_{w\in\cw^{\minu}_{\sig}}
\Bigl[ L\bigl( w^{-1}\hat{\rho} -\hat{\rho}\bigr) \otimes
L\bigl( w^{-1}\hat{\rho}-\hat{\rho}\bigr)^*\Bigr]^{\hat{\fr}_{\sig}} .
  \eeqn

Now, by a similar argument to that given in [K3, Section 2.4], the proof of
Theorem 3.1 follows.  We omit the details.
  \end{proof}

 \section{Structure of the Algebra $A^{\fk}$}

Let $G$ be a connected, simply-connected complex algebraic group with Lie algebra $\fg$.  The involution $\sig$ of $\fg$, of course, induces an involution of
$G$.  Choose a maximal compact subgroup $G_o$ of $G$ which is stable under $\sig$
 and such that the subgroup $K_o := G^{\sig}_o$ of $\sig$-invariants is a maximal
 compact subgroup of $K := G^{\sig}$ (cf. [He]).
 Moreover, as is well known, $K$ is connected and hence so is $K_o$.

Let $\Om^{\sig}(G_o)$ be the space of all continuous maps $f: S^1 \to G_o$ which are $\sig$-equivariant, i.e.,
  \[
f(-z) = \sig (f(z))\qquad\text{  for all }z\in S^1 .
  \]
We put the compact-open topology on $\Om^{\sig}(G_o)$.  Clearly, the subspace
of constant loops can be identified with $K_o$.  Equivalently, we can view
$\Om^{\sig}(G_o)$ as the space of continuous maps $\bar{f}: [0, 2\pi ]
\to G_o$ such that
  \[
\bar{f}(t+\pi ) = \sig (\bar{f}(t)), \qquad\text{  for all }0\leq t\leq \pi .
  \]

In particular, $\bar{f}(2\pi ) = \sig^2(\bar{f}(0)) = \bar{f}(0)$.  The correspondence $f\rtsquig \bar{f}$ is given by $\bar{f}(t) = f(e^{it})$, for $0\leq t\leq 2\pi$.

Consider the fibration
  \[
\Om^{\sig}_1(G_o) \to \Om^{\sig}(G_o)/K_o \overset{\gamma}\longrightarrow G_o/K_o ,
  \]
where $\gamma(fK_o) = f(1)K_o$ for $f\in\Om^{\sig}(G_o)$ and $\Om_1^{\sig}(G_o)$
is the subspace of $\Om^{\sig}(G_o)$ consisting of those $f$ such that $f(1)=1$.

Of course, $\Om_1^{\sig}(G_o)$ can be identified with the based loop space $\Om_1(G_o)$ of $G_o$ under $f\rtsquig \bar{f}_{|_{[0, \pi ]}}$.

Define the $\fk$-module map $\bar{c} : \fk^* \to \wed^2(\fp )^*$ by $(\bar{c} f)(x\wed y) = f([x,y])$,
for $x,y\in\fp$.  This gives rise to a map (still denoted by)
  \[   \bar{c} : S(\fk^*) \to \wed (\fp )^* .
  \]
Consider the restriction of $\bar{c}$ to the subring of $\fk$-invariants
  \[
c: S(\fk^*)^{\fk} \to C(\fg ,\fk )\simeq [\wed (\fp )^*]^{\fk} .
  \]

Then, the map $c$ is the Chern-Weil homomorphism with respect to a
$G_o$-invariant connection on the $G_o$-equivariant principle $K_o$-bundle
$G_o\to G_o/K_o$.

Observe that since $\fk$ is the $+1$ eigenspace of an involution of $\fg$,
the differential $\del \equiv 0$ on $C^*(\fg ,\fk )$.  Thus,
  \[
C^*(\fg ,\fk ) \simeq H^*(\fg ,\fk ) \simeq H^*(G_o/K_o).
  \]

Thus, in our case, we can think of $c$ as the map $c: S(\fk^*)^{\fk}
\to H^*(\fg ,\fk ) \simeq H^*(G_o/K_o)$.

We now recall the following result due to H. Cartan on the cohomology of
$G_o/K_o$ with complex coefficients (cf. [C, \S10]).

  \begin{theorem}  There exists a finite-dimensional graded subspace
$V\subset H^*(G_o/K_o)$ concentrated in odd degrees such that, as graded
algebras,
  \[
H^*(G_o/K_o) \simeq \wed(V)\otimes \Imo c.
  \]
  \end{theorem}

  \begin{corollary}  Consider the map $\gam : \Om^{\sig}(G_o)/K_o \to G_o/K_o$
defined earlier (obtained from the evaluation at 1).  Then, the induced
map in cohomology
  \[
\gam^*: H^*(G_o/K_o) \to H^*(\Om^{\sig}(G_o)/K_o)
  \]
under the identification
  \[
H^*(G_o/K_o) \simeq \wed (V)\otimes \Imo c
  \]
of the above theorem, satisfies
  \[   \gam^*_{|V} \equiv 0.
  \]
In particular, $\Imo (\gam^*) = \gam^*(\Imo c)$.
  \end{corollary}

  \begin{proof}  This follows immediately from the fact that
$H^*(\Om^{\sig}(G_o)/K_o)$ is concentrated in even degrees only and $V$
lies in odd cohomological degrees.
  \end{proof}

  Let  $L^{\sig}(\fg )$ be the twisted loop algebra $\bigoplus_{i\in\bz}\fg_i\otimes t^i$, i.e.,
  $L^{\sig}(\fg )$ is the space of all algebraic maps
$f: \bc^*\to\fg$ satisfying $f(-z) =\sig (f(z))$ for all $z\in\bc^*$ and the Lie
algebra structure is obtained by taking the pointwise bracket.  This is a  subalgebra
of the loop algebra
\[L(\fg):=\fg\otimes \Bbb C[t,t^{-1}].\]
  Let $L^{\sig}_1(\fg )$ be the kernel of the evaluation map $L^{\sig}(\fg )
\to \fg$ at 1, $x\otimes a(t)\mapsto a(1)\, x$.  Similarly, by
 $L^{\sig}_1(G_o )$, we mean the set of algebraic maps $f: S^1\to G_o$ with
$f(-z) =\sig (f(z))$ for all $z\in S^1$ and $f(1)=1$ (where we call a map
$f: S^1\to G_o$ {\em algebraic} if it extends to an algebraic map $\bar{f}:
\bc^*\to G$).

We recall the following result from [K1, Theorem 1.6].

  \begin{theorem}  Appropriately defined, the integration map defines an
algebra isomorphism in cohomology
  \[
 H^*(L^{\sig}(\fg ), \fk ) \simeq H^*(\cx_{\sig}).
  \]
Similarly, we have an algebra isomorphism
  \[
H^*(L^{\sig}_1(\fg )) \simeq H^*(L_1^{\sig}(G_o)),
  \]
   where  $L^{\sig}_1(G_o)$ is
endowed with the Hausdorff topology induced from an ind-variety structure.
  \end{theorem}

Analogous to the result of Garland-Raghunathan [GR], we have the following.

  \begin{theorem}  The inclusion  $L^{\sig}_1(G_o) \hookrightarrow
\Om_1^{\sig}(G_o)$ is a homotopy equivalence, where  $L^{\sig}_1(G_o)$ is
endowed with the Hausdorff topology as in the previous theorem
and $\Om^{\sig}_1(G_o)$ is equipped with the compact-open topology.

Similarly, the projective ind-variety $\cx_{\sig}$ under the Hausdorff topology is homotopically equivalent with the space $\Om^{\sig}(G_o)/K_o$.
  \end{theorem}

 For any
invariant homogeneous polynomial $P\in S^{d+1}(\fg^*)^{\fg}$ of degree $d+1$ ($d\geq 1$), define the map
  \[
\phi_P: \wed^{2d}_{\bc}(L(\fg )) \to \bc
  \]
by
  \[
\phi_P(v_0\wed v_1\wed \cdots \wed v_{2d-1}) = \frac{1}{\pi i} \int^{\pi}_{\theta =0} \Phi_P (v_0\wed v_1\wed\cdots\wed v_{2d-1}),
  \]
where $\Phi_P: \wed_{\bc}^{2d} (L(\fg )) \to\Om^1$ is the map defined by
  \begin{multline*}
\Phi_P\bigl( v_0\wed v_1\wed\cdots\wed v_{2d-1}\bigr) :=
\sum_{\mu\in S_{2d}}\eps (\mu )\, P\Bigl( v_{\mu (0)}, \bigl[ v_{\mu (1)},
v_{\mu (2)}\bigr], \\
\dots , \bigl[ v_{\mu (2d-3)}, v_{\mu (2d-2)}\bigr] ,
dv_{\mu (2d-1)}\Bigr) .
   \end{multline*}
Here $\Om^1$ is the space of algebraic 1-forms on $\bc^*$,
$d(x\otimes a(t)) = x\otimes a'(t)dt$ (for $x\in\fg$ and $a(t)\in\bc
[t,t^{-1}]$) and in the  integral $\int^{\pi}_{\theta =o}$
we make the substitution $t=e^{i\theta}$.

Let $\pi_{\fk}: \fg\to\fk$ be the projection under the decomposition
$\fg = \fk\oplus\fp$.  We similarly define $\pi_{\fp}$.  Define the
$\fk$-invariant map (for any $P\in S^{d+1}(\fg^*)^{\fg}$)
  \[
\hat{\phi}_P: \wed^{2d}_{\bc}\bigl( L^{\sig}(\fg )/\fk\bigr) \to \bc
  \]
by
  \[
\hat{\phi}_P(\bar{v}_0\wed \cdots \wed\bar{v}_{2d-1}) =
\phi_P(v^o_0\wed\cdots\wed v^o_{2d-1}),
  \]
where $\bar{v}_i := v_i+\fk\in L^{\sig}(\fg )/\fk$ and $v^o_i :=
v_i-v_i(1)$.  Then, $\hat{\phi}_P$ can be viewed as a
cochain for the Lie algebra pair $(L^{\sig}(\fg ), \fk )$.

  \begin{lemma}  Let $P\in S^{d+1}(\fg^*)^{\fg}$.  Then, for the
differential $\del$ in the standard cochain complex of the pair
$(L^{\sig}(\fg ), \fk )$, $\del\hat{\phi}_P$ descends to a cocycle for
the Lie algebra pair $(\fg ,\fk )$ under the evaluation map
$L^{\sig}(\fg )\to\fg$ at 1.
  \end{lemma}

  \begin{proof}
Observe first that, by [FT], the following diagram is commutative up
to a nonzero scalar multiple (i.e., $d\circ \beta_P = z^{-1}\Phi_P\part$,
for some $z\in\bc^*$).
  \[  \begin{CD}
\wed^{2d+1}_{\bc}(L(\fg )) @>{\beta_P}>> \Om^0 \\
@VV{\part}V @VV{d}V \\
\wed^{2d}_{\bc}(L(\fg )) @>>{\Phi_P}> \Om^1 ,
  \end{CD}  \]
where
  \begin{multline*}
\beta_P(v_0\wed\cdots\wed v_{2d}) := \sum_{\mu\in S_{2d+1}}
\eps (\mu
)\, P\Bigl( v_{\mu (0)}, \bigl[ v_{\mu (1)}, v_{\mu (2)}\bigr] ,\\
\dots , \bigl[ v_{\mu (2d-1)}, v_{\mu (2d)}\bigr]\Bigr) ,
  \end{multline*}
$\Om^0$ is the space of algebraic functions on $\bc^*$, $d$ is the
standard deRham differential, and $\part$ is the standard
differential in the chain complex of the Lie algebra $L(\fg
)$.  Thus, for $v_i\in L(\fg )$,
  \begin{align}
(\del\phi_P)(v_0\wed v_1\wed \cdots\wed v_{2d}) &= \frac{1}{\pi
i}\int^{\pi}_{\theta =0} \Phi_P\bigl(\part (v_0\wed v_1\wed \cdots\wed
v_{2d})\bigr) \notag\\
&= \frac{z}{\pi i}\int^{\pi}_{\theta =0}  d\bigl(\beta_P (v_0\wed
v_1\wed \cdots\wed v_{2d})\bigr) \notag\\
&= \frac{z}{\pi i}\bigl( \beta_P (v_0(-1)\wed\dots\wed v_{2d}(-1)) \notag\\
&\qquad\qquad -\beta_P (v_0(1)\wed\cdots\wed v_{2d}(1))\bigr).
  \end{align}
We next show that for any $v_0,\dots ,v_{2d}\in L^{\sig}(\fg )$,
  \beqn
(\del \hat{\phi}_P) (\bar{v}_0\wed\cdots\wed \bar{v}_{2d}) = (\del
\phi_P) (v^o_0\wed\cdots\wed v^o_{2d}),
  \eeqn
where $\bar{v}_i$ and $v^o_i$ are defined above  the statement of
this lemma. For any $x,y\in L(\fg )$,
  \beqn
[x,y]^o - [x^o,y^o] = [x(1), y^o] + [x^o, y(1)].
  \eeqn
  Thus,
  \begin{align*}
(\del\hat{\phi}_P) &(\bar{v}_0\wed \cdots\wed\bar{v}_{2d}) - \del\phi_P (v^o_0\wed
\cdots\wed v^o_{2d})\\
&= \sum_{i<j}(-1)^{i+j}\, \phi_P\Bigl( \bigl( [v_i,v_j]^o - [v^o_i,v^o_j]\bigr)\wed v^o_0\wed\cdots\wed \widehat{v^o_i}\wed\cdots\\
  &\qquad\qquad \wed \widehat{v^o_j}\wed \cdots\wed v^o_{2d}\Bigr)\\
&=  \sum_{i<j}(-1)^{i+j}\, \phi_P\Bigl( \bigl( [v_i(1),v^o_j] +
[v^o_i, v_j(1)] \bigr)\\
&\qquad\qquad \wed v^o_0\wed\cdots\wed
\widehat{v^o_i}\wed\cdots \wed \widehat{v^o_j}\wed \cdots\wed v^o_{2d}\Bigr),
\text{  by (3)}\\
&= \sum_{i<j}(-1)^{i+j}\, \phi_P\bigl( [v_i(1), v^o_j]\wed
v^o_0\wed\cdots\wed \widehat{v^o_i}\wed\cdots \wed\widehat{v^o_j}\wed \cdots\wed v^o_{2d}\bigr) \\
&\qquad + \sum_{i>j}(-1)^{i+j}\, \phi_P\bigl( [v^o_j, v_i(1)]\wed
v^o_0\wed\cdots\wed \widehat{v^o_j}\wed\cdots \wed\widehat{v^o_i}\wed \cdots\wed v^o_{2d}\bigr)\\
&= \sum_i (-1)^i \bigl(v_i(1)\cdot \phi_P\bigr)
\bigl( v^o_0\wed \cdots\wed\widehat{v^o_i}\wed\cdots\wed v^o_{2d}\bigr)\\
&= 0, \qquad\text{ since $\phi_P$ is $\fg$-invariant.}
   \end{align*}
This proves (2).

In particular, for any $v_0\in L^{\sig}(\fg )$ such that $v_0(1)=0$ and
$v_1, \dots, v_{2d} \in L^{\sig}(\fg )$, we get
  \begin{align*}
 \del \hat{\phi}_P\bigl( \bar{v}_0\wed\cdots\wed\bar{v}_{2d}\bigr) &=
\del\phi_P\bigl( v_0\wed v^o_1\wed\cdots\wed v^o_{2d}\bigr)\\
&=\frac{z}{\pi i}\beta_P\bigl( v_0(-1)\wed v^o_1(-1)\wed\cdots\wed v^o_{2d}(-1)\bigr),\,\,
\text{since}\, v_0(1)=0\\
&= 0, \,\,
\text{since}\, v_0(-1)=\sigma(v_0(1))=0.
    \end{align*}
This proves the lemma.
 \end{proof}

By Identity (1) of the above lemma, the restriction $\bar{\phi}_P$ of $\phi_P$
to $\wed_{\bc}^{2d}(L_1^{\sig}(\fg ))$ is a cocycle (for the Lie algebra $L_1^{\sig}(\fg )$).

As is well known, $S(\fg^*)^{\fg}$ is freely generated by certain homogeneous
polynomials $P_1,\dots ,P_{\ell_\fg}$ of degrees $m_1+1, m_2+1, \dots ,m_{\ell_\fg}+1$ respectively,
where $\ell_\fg$ is the rank of $\fg$ and $m_1=1 < m_2\leq \cdots \leq m_{\ell_\fg}$
 are the exponents of $\fg$.

The following result is obtained by combining [PS, Proposition 4.11.3] and Theorems 4.3 and 4.4.

  \begin{theorem}  The cohomology classes $[\bar{\phi}_{P_1}], \dots ,[\bar{\phi}_{P_{\ell_\fg}}] \in H^*(L_1^{\sig}(\fg ))$ freely generate the algebra
  \[
H^*\bigl( L_1^{\sig}(\fg )\bigr) \simeq H^*\bigl( L_1^{\sig}(G_o)\bigr)
\simeq H^*\bigl(\Om_1^{\sig}(G_o)\bigr) .
  \]
  \end{theorem}

Define the differential graded algebra (for short DGA)
  \[
\cd = H^*(L_1^{\sig}(\fg )) \otimes C^*(\fg ,\fk )
  \]
under the graded tensor product algebra structure.  We define the differential
$d$ in $\cd$ as follows:  Take $d_{|_{C^*(\fg ,\fk )}}$ as
the standard differential $\del$ of the cochain complex $C^*(\fg ,\fk )$ of
the Lie algebra pair $(\fg ,\fk )$ and $d([\bar{\phi}_{P_i}]) = \del\hat{\phi}_{P_i}$
(cf. Lemma 4.5).  There is a differential graded algebra homomorphism
$\mu : \cd\to C^*(L^{\sig}(\fg ),\fk )$ defined by
  \[
\mu\bigl( [\bar{\phi}_{P_i}]\bigr) =\hat{\phi}_{P_i}
  \]
 and $\mu_{|_{C^*(\fg ,\fk )}}$ is the canonical inclusion
$j: \, C^*(\fg ,\fk )\subset C^*(L^{\sig}(\fg ),\fk )$ under the evaluation
map at 1.

Applying the Hirsch lemma to the fibration
\[\Om^{\sig}_1(G_o)
\to \Om^{\sig}(G_o)/K_o\overset{\gamma}{\longrightarrow} G_o/K_o,\] (cf.
[DGMS, Lemma 3.1]), and using Theorems 4.3, 4.4 and 4.6,
we get the following.

  \begin{theorem}  The map $\mu$ induces a graded algebra isomorphism in
cohomology
  \[
[\mu ]: H^*(\cd ) \simto H^*(\cx_{\sig}).
  \]
In particular, by Corollary 4.2, any cohomology class $[x]\in H^*(\cx_{\sig})$ can be
represented by a cocycle $x\in C^*(L^{\sig}(\fg ),\fk )$ of the form
  \[
x = \sum_{\fri =(i_1,\dots ,i_{\ell_\fg})\in\bz^{\ell_\fg}_+} j\bigl(
c(Q_{\fri})\bigr)
(\hat{\phi}_{P_1})^{i_1}\cdots (\hat{\phi}_{P_{\ell_\fg}})^{i_{\ell_\fg}} ,
  \]
for some $Q_{\fri}\in S(\fk^*)^{\fk}$, where $c: S(\fk^*)^{\fk}\to C(\fg ,
\fk )$ is the Chern-Weil homomorphism defined in the beginning of this section.
  \end{theorem}

Finally, we are ready to prove the second main theorem of this paper.

  \begin{theorem}  Let $\fg$ be a simple Lie algebra and let $\sig$ be an
involution of $\fg$ with $+1$ (resp. $-1$) eigenspace $\fk$ (resp. $\fp$).
Assume that $\fp$ is an irreducible $\fk$-module.  Then, the algebra
$A^{\fk}$ of $\fk$-invariants of $A$ is generated (as an algebra) by the element $S$,
where $A$ and $S$ are defined in the Introduction.

In particular, $(A^{\fk})^{p,q}=0$ if $p\neq q$.
  \end{theorem}

  \begin{proof}  By Theorem 3.1, the algebra $B^{\fk}$ is graded isomorphic
  with the singular cohomology $H^*(\cy )$, where
  $B := \frac{\wed (\fp )\otimes\wed (\fp )}{\ip<C_1\oplus C_2>}$.  Moreover, the
  inclusion $a: \cy\subset \cx_{\sig}$ induces a surjection in cohomology,
  since $\cx_{\sig}$ is obtained from $\cy$ by attaching real even-dimensional cells
  (by virtue of the Bruhat decomposition).  Thus, we have
  \[
H^*(\cx_{\sig}) \overset{a^*}{\twoheadrightarrow} H^*(\cy ) \overset{\xi}{\simto} B^{\fk }
\overset{\eta}{\longrightarrow} A^{\fk},
  \]
where $\eta: B^{\fk}\to A^{\fk}$ is the standard quotient map.

By Theorem 4.7, any cohomology class $[x]\in H^*(\cx_{\sig})$ can be
represented by a cocycle $x\in C^*(L^{\sig}(\fg ),\fk )$ of the form
  \[
x=\sum_{\fri =(i_1,\dots ,i_{\ell_\fg})\in\bz^{\ell_\fg}_+} j\bigl(
c(Q_{\fri})\bigr)
(\hat{\phi}_{P_1})^{i_1}\cdots (\hat{\phi}_{P_{\ell_\fg}})^{i_{\ell_\fg}},
  \]
for some $Q_{\fri}\in S(\fk^*)^{\fk}$.

If $Q_{\fri}$ has constant term 0, from the definition of the Chern-Weil
homomorphism $c$, it is clear that under the composite map $\fy :=
\eta\circ\xi\circ a^*$, $j(c(Q_{\fri}))$ goes to zero.  Further, by an argument
similar to the proof
of Theorem 2.8 in [K3], we
see that $\hat{\phi}_{P_i}$ goes to zero under $\fy$ for any $2\leq i\leq\ell_\fg$.
We briefly recall the main argument here.

For any $\mu\in S_{2d}$ and $P\in S^{d+1}(\fg^*)^\fg$ ($d\geq 2$), consider the
linear form
  \[
\hat{\phi}_{P,\mu}: \otimes^{2d}_\bc \bigl(L^\sigma(\fg)/\fk\bigr) \to \bc ,
  \]
defined by
 \begin{multline*} \hat{\phi}_{P,\mu} \Bigl( \bar{v}_0\otimes \bar{v}_1\otimes
\cdots\otimes \bar{v}_{2d-1}\Bigr)\\
  = \int^{\pi}_{\theta =0} P\Bigl( v^o_{\mu (0)}, \bigl[ v^o_{\mu
(1)}, v^o_{\mu (2)}\bigr] , \ldots, \bigl[ v^o_{\mu(2d-3)}, v^o_{\mu (2d-2)}
\bigr] , dv^o_{\mu (2d-1)}\Bigr),
  \end{multline*}
where $\bar{v}_i:=v_i+\fk$.
For the notational convenience, assume $\mu (1) < \mu (2)$.  For any fixed
$$v_0,v_1, \ldots, \hat{v}_{\mu (1)}, \ldots, \hat{v}_{\mu (2)},
\ldots, v_{2d-1} \in L^\sigma(\fg),$$
 consider the restriction
$\bar{\phi}_{P,\mu}$ of the function $\hat{\phi}_{P,\mu}$ to
$$\bar{v}_0\x \bar{v}_1\x
\cdots\x \oplus_{p\in \bz} \,\fg_{2p+1}\otimes t^{2p+1} \x\cdots\x \oplus_{p\in \bz}
 \,\fg_{2p+1}\otimes t^{2p+1}\x\cdots\x \bar{v}_{2d-1},$$
 where
the two copies of $\oplus_{p\in \bz}\, \fg_{2p+1}\otimes t^{2p+1}$ are placed in
the $\mu (1)$ and $\mu
(2)$-th slots.  Then, under the identification $\fg_p\otimes t^p\cong
(\fg_p\otimes t^p)^*$ induced from the bilinear form $\ip<\,\, , \,\,>,$
  \begin{align*}
\bar{\phi}_{P,\mu} &= \sum_{i,j,m,n} f_i(n) \otimes f_j(m)\,  \int^{\pi}_{\theta =0} P\Bigl(
v^o_{\mu (0)}, \bigl[ e_i(n)^o,e_j(m)^o\bigr] ,\bigl[ v^o_{\mu
(3)}, v^o_{\mu (4)}\bigr] ,\ldots,\\
&\hspace{1.8in}\bigl[ v^o_{\mu (2d-3)}, v^o_{\mu (2d-2)} \bigr] , dv^o_{\mu
(2d-1)}\Bigr) \\
  &= \sum_{i,j,m,n,k'} f_i(n)\otimes f_j(m)\,  \int^{\pi}_{\theta =0} P\bigl( -, \ip<[e_i,e_j],
e'_{k'}> F_{k'}{(n,m)}, -\bigr)\\
 &= \sum_{i,j,m,n,k'} \ip<e_i, [e_j,e'_{k'}]>\, f_i(n)\otimes f_j(m)\,  \int^{\pi}_{\theta =0}
 P\bigl(
-, F_{k'}{(n,m)}, -\bigr) \\
 &= \sum_{j,k',m,n} [e_j,e'_{k'}](n) \otimes f_j(m)\,  \int^{\pi}_{\theta =0}
 P\bigl( -, F_{k'}{(n,m)},
-\bigr)\\
 &= - \sum_{j,k',m,n} [e'_{k'},e_j](n) \otimes f_j(m)\,  \int^{\pi}_{\theta =0}
 P\bigl( -, F_{k'}{(n,m)},
-\bigr),
\end{align*}
 where, as in the Introduction, $\{e_i\}$ is a basis of $\fp$ and $\{f_i\}$ is the
 dual basis; $\{e'_{k'}\}$ is a basis of $\fk$ and $\{f'_{k'}\}$ is the
 dual basis; $m,n$ run over the odd integers and $F_{k'}(n,m):=f'_{k'}(n+m)-f'_{k'}(n)-f'_{k'}(m)+f'_{k'}$.

Thus, only the powers of $\hat{\phi}_{P_1}$ contribute to the image of
$\fy$.  This completes the proof of the theorem.
  \end{proof}

 \begin{remark}  It is likely that for the validity of Theorem 4.8 it is enough
to assume that $\fg$ is semisimple (not necessarily simple).  However, we must
assume that $\fp$ is $\fk$-irreducible under the adjoint action since
the second grade component $(A^2)^{\fk}$ has dimension at least equal to
the number of irreducible components of the $\fk$-module $\fp$.
  \end{remark}

  \end{document}